\documentclass[11pt]{amsart}
\usepackage{amssymb}
\usepackage{amsthm}
\usepackage{amscd}
\usepackage{array}
\usepackage{mathscinet}
\usepackage[sort,numbers]{natbib}

\usepackage{tikz}
\usetikzlibrary{decorations.pathreplacing,decorations,decorations.pathmorphing}
\usetikzlibrary{patterns}
\newtheorem{theorem}{Theorem}[section]
\newtheorem{lemma}[theorem]{Lemma}

\newtheorem{question}[theorem]{Question}
\newtheorem*{poincare}{Poincar\'e Recurrence Theorem}
\newtheorem*{birkhoff}{Birkhoff's Ergodic Theorem}

\theoremstyle{definition}
\newtheorem{definition}[theorem]{Definition}

\makeatletter

\def\dotminussym#1#2{%
  \setbox0=\hbox{$\m@th#1-$}%
  \kern.5\wd0%
  \hbox to 0pt{\hss\hbox{$\m@th#1-$}\hss}%
  \raise.6\ht0\hbox to 0pt{\hss$\m@th#1.$\hss}%
  \kern.5\wd0}

\mathchardef\mhyphen="2D

\allowdisplaybreaks[2]

\newcommand{\ml}{Martin-L\"{o}f }

\newcommand{\reals}{2^\omega}
\newcommand{\strings}{2^{<\omega}}

\newcommand{\re}{c.e.\ }

\begin{document}

\title{Randomness and Non-ergodic Systems}
\date{\today}

\author{Johanna N.Y.\ Franklin}
\address{Department of Mathematics, University of Connecticut U-3009, 196 Auditorium Road, Storrs, CT 06269-3009, USA}
\email{johanna.franklin@uconn.edu}
\urladdr{www.math.uconn.edu/~franklin}

\author{Henry Towsner}
\address {Department of Mathematics, University of Connecticut U-3009, 196 Auditorium Road, Storrs, CT 06269-3009, USA}
\thanks{The second author was partially supported by NSF grant DMS-1157580.}
\email{htowsner@gmail.com}
\urladdr{www.math.uconn.edu/~towsner}

\begin{abstract}
We characterize the points that satisfy Birkhoff's ergodic theorem under certain computability conditions in terms of algorithmic randomness. First, we use the method of cutting and stacking to show that if an element $x$ of the Cantor space is not \ml random, there is a computable measure-preserving transformation and a computable set that witness that $x$ is not typical with respect to the ergodic theorem, which gives us the converse of a theorem by V'yugin. We further show that if $x$ is weakly 2-random, then it satisfies the ergodic theorem for all computable measure-preserving transformations and all lower semi-computable functions.
\end{abstract}

\maketitle
\section{Introduction}

Random points are typical with respect to measure in that they have no measure-theoretically rare properties of a certain kind, while ergodic theorems describe regular measure-theoretic behavior. There has been a great deal of interest in the connection between these two kinds of regularity recently. We begin by defining the basic concepts in each field and then describe the ways in which they are related. Then we present our results on the relationship between algorithmic randomness and the satisfaction of Birkhoff's ergodic theorem for computable measure-preserving transformations with respect to computable (and then lower semi-computable) functions.  Those more familiar with ergodic theory than computability theory might find it useful to first read Section \ref{ergodic_people}, a brief discussion of the notion of algorithmic randomness in the context of ergodic theory.

\subsection{Algorithmic randomness in computable probability spaces}

For a general reference on algorithmic randomness, see \cite{dhbook,dhnt,niesbook}. We will confine our attention to the Cantor space $\reals$ with the Lebesgue measure $\lambda$.  In light of Hoyrup and Rojas' theorem that any computable probability space is isomorphic to the Cantor space in both the computable and measure-theoretic senses \cite{hr09}, there is no loss of generality in restricting to this case.

% When $\sigma\in 2^{<\omega}$, we write $[\sigma]\subseteq 2^\omega$ for the set of infinite sequences extending with $\sigma$.  To define $\lambda$, it suffices to have $\lambda([\sigma])=2^{-|\sigma|}$ for all $\sigma\in 2^{<\omega}$.  When $V\subseteq 2^{<\omega}$, we write $[V]=\cup_{\sigma\in V}[\sigma]$.

We present Martin-L\"of's original definition of randomness \cite{ml66}.
\begin{definition}
 An effectively \re sequence $\langle V_i\rangle$ of subsets of $\strings$ is a \emph{\ml test} if $\lambda([V_i])\leq 2^{-i}$ for every $i$. If $x\in\reals$, we say that $x$ is \emph{\ml random} if for every \ml test $\langle V_i\rangle$, $x\not\in\cap_i [V_i]$.
\end{definition}

It is easy to see that $\lambda(\cap_i[V_i])=0$ for any \ml test, and since there are only countably many \ml tests, almost every point is \ml random.

In Section \ref{upcross}, we will also consider weakly 2-random elements of the Cantor space. Weak 2-randomness is a strictly stronger notion than \ml randomness and is part of the hierarchy introduced by Kurtz in \cite{kurtz}.

\begin{definition}
 An effectively \re sequence $\langle V_i\rangle$ of subsets of $\strings$ is a \emph{generalized \ml test} if $\lim_{n\rightarrow\infty} \lambda([V_i]) = 0$. If $x\in\reals$, we say that $x$ is \emph{weakly 2-random} if for every generalized \ml test $\langle V_i\rangle$, $x\not\in\cap_i [V_i]$.
\end{definition}

%We observe that nothing about this definition relies on a particular property of the Cantor space with Lebesgue measure: As long as we can define uniformly \re sequences of open sets, we can define \ml randomness. This leads us to the definition of a computable probability space. 

%A complete metric space $(X,d)$ is computable if there is a computable enumeration $\langle q_i\rangle$ of a countable dense subset of $X$ such that the function mapping each pair $(i,j)$ to $d(q_i,q_j)$ is computable. This allows us to enumerate a list $\langle U_i\rangle$ of basic open sets defined by these ``rational'' balls. We say that an \re open set is a union of an \re set of these basic open sets and that a probability measure is computable if the measure of a finite union of basic open sets is a left-\re real. Finally, a computable probability space is a computable metric space whose measure is a computable probability measure. It is clear that the definition of \ml randomness transfers to any such space. In fact, Hoyrup and Rojas have shown that .

\subsection{Ergodic theory}

Now we discuss ergodic theory in the general context of an arbitrary probability space before transferring it to the context of a computable probability space. The following definitions can be found in \cite{halmos}.

\begin{definition}
 Suppose $(X,\mu)$ is a probability space, and let $T:X\rightarrow X$ be a measurable transformation.
\begin{enumerate}
 \item $T$ is \emph{measure preserving} if for all measurable $A\subseteq X$, $\mu(T^{-1}(A)) = \mu(A)$.
\item A measurable set $A\subseteq X$ is \emph{invariant} under $T$ if $T^{-1}(A)=A$ modulo a set of measure 0.
\item $T$ is \emph{ergodic} if it is measure preserving and every $T$-invariant measurable subset of $X$ has measure 0 or measure 1.
\end{enumerate}
\end{definition}

One of the most fundamental theorems in ergodic theory is Birkhoff's Ergodic Theorem:
%were originally proven by Poincar\'e and Birkhoff. 
 
\begin{birkhoff} \cite{birkhoff31}
 Suppose that $(X,\mu)$ is a probability space and $T:X\rightarrow X$ is measure preserving.  Then for any $f\in L_1(X)$ and almost every $x\in X$,
\[\lim_{n\rightarrow\infty} \frac{1}{n} \sum_{i<n}f(T^i(x))\]
converges.  Furthermore, if $T$ is ergodic then for almost every $x$ this limit is equal to $\int f\; d\mu$.
\end{birkhoff}

If we restrict ourselves to a countable collection of functions, this theorem gives a natural notion of randomness---a point is random if it satisfies the conclusion of the ergodic theorem for all functions in that collection.  In a computable measure space, we can take the collection of sets defined by a computability-theoretic property and attempt to classify this notion of randomness in terms of algorithmic randomness.  In particular, we are interested in the following property:

\begin{definition}
 Let $(X,\mu)$ be a computable probability space, and let $T:X\rightarrow X$ be a measure-preserving transformation.  Let $\mathcal{F}$ be a collection of functions in $L_1(X)$.
%\begin{enumerate}
% \item A point $x\in X$ is a Poincar\'e point for $T$ with respect to $\mathcal{C}$ if for every $E\in\mathcal{C}$ with positive measure, $T^n(x)\in E$ for infinitely many $n$.
A point $x\in X$ is a \emph{weak Birkhoff point} for $T$ with respect to $\mathcal{F}$ if for every $f\in\mathcal{F}$,
\[\lim_{n\rightarrow\infty} \frac{1}{n} \sum_{i<n}f(T^i(x))\]
converges.  $x$ is a \emph{Birkhoff point} if additionally
\[\lim_{n\rightarrow\infty} \frac{1}{n} \sum_{i<n}f(T^i(x))=\int f\; d\mu.\]
%\end{enumerate}
\end{definition}
The definition of a Birkhoff point is only appropriate when $T$ is ergodic; when $T$ is nonergodic, the appropriate notion is that of a weak Birkhoff point.

There are two natural dimensions to consider: the ergodic-theoretic behavior of $T$ and the algorithmic complexity of $\mathcal{C}$.  The case where $T$ is ergodic has been largely settled.

A point is \ml random if and only if the point is Birkhoff for all computable ergodic transformations with respect to lower semi-computable functions \cite{bdms10, fgmn}.  The proof goes by way of a second theorem of ergodic theory:
 \begin{poincare} [\cite{poincare}, Chapter 26]
Suppose that $(X,\mu)$ is a probability space and $T:X\rightarrow X$ is measure preserving. Then for all $E\subseteq X$ of positive measure and for almost all $x\in X$, $T^n(x)\in E$ for infinitely many $n$.
 \end{poincare}
In short, the Poincar\'e Recurrence Theorem says that an ergodic transformation $T$ returns almost every point to every set of positive measure repeatedly, and Birkhoff's Ergodic Theorem says that it will do so with a well-defined frequency in the limit.

A point $x\in X$ is a \emph{Poincar\'e point} for $T$ with respect to $\mathcal{C}$ if for every $E\in\mathcal{C}$ with positive measure, $T^n(x)\in E$ for infinitely many $n$.  In \cite{k85}, Ku\v{c}era proved that a point in the Cantor space is \ml random if and only if it is a Poincar\'e point for the shift operator with respect to effectively closed sets. Later, Bienvenu, Day, Mezhirov, and Shen generalized this result and showed that in any computable probability space, a point is \ml random if and only if it is a Poincar\'e point for computable ergodic transformations with respect to effectively closed sets \cite{bdms10}.  The proof that \ml random points are Poincar\'e proceeds by showing that a point which is Poincar\'e for any computable ergodic transformation with respect to effectively closed sets must also be a Birkhoff point for computable ergodic transformations with respect to lower semi-computable functions \cite{bdms10, fgmn}.

Similarly, G\'acs, Hoyrup, and Rojas have shown that a point is \emph{Schnorr random} if and only if the point is Birkhoff for all computable ergodic transformations with respect to computable functions \cite{ghr11}.  (Recall that $x$ is Schnorr random if $x\not\in\cap_i [V_i]$ for all \ml tests $\langle V_i\rangle$ where $\lambda([V_i])=2^{-i}$ for all $i$; Schnorr randomness is a strictly weaker notion than \ml randomness \cite{schnorr}.)  They also consider the case where there are strong mixing assumptions on $T$ in addition to being ergodic and show that the equivalence with Schnorr randomness still holds.

In this paper, we consider the analogous situations when $T$ is nonergodic.  V'yugin \cite{v97} has shown that if $x\in\reals$ is \ml random then $x$ is weakly Birkhoff for any (not necessarily ergodic) computable measure-preserving transformation $T$ with respect to computable functions.  Our main result is the converse: that if $x$ is not \ml random then $x$ is not weakly Birkhoff for some particular transformation $T$ with respect to computable functions (in fact, with respect to computable sets).

These results can be summarized in Table \ref{table:summary}.

\begin{table}
\setlength{\extrarowheight}{1.5pt}
\begin{tabular}{c|c|c}
 & \multicolumn{2}{c}{Transformations} \\
 Sets & Ergodic & Nonergodic\\
\hline
Computable& Schnorr & \ml\\
 &\cite{ghr11} & \cite{v97}+Theorem \ref{main_thm}\\
\hline
Lower semi-computable &\ml&?\\
& \cite{bdms10, fgmn}&\\

\end{tabular}
\label{table:summary}
\end{table}

This says that a point is weakly Birkhoff for the specified family of computable transformations with respect to the specified collection of functions if and only if it is random in the sense found in the corresponding cell of the table.

We also begin an analysis of the remaining space in the table; we give an analog of V'yugin's result, showing that if $x$ is weakly $2$-random then $x$ is a weak Birkhoff point for all computable measure-preserving transformations with respect to lower semi-computable functions. 

The next two sections will be dedicated to a discussion of the techniques we will use in our construction. Section \ref{notation} contains a description of the type of partial transformations we will use to construct the transformation $T$ mentioned above, and Section \ref{lemmas} discusses our methods for building new partial transformations that extend other such transformations. We combine the material from these two sections to prove our main theorem in Section \ref{main}, while Section \ref{upcross} contains a further extension of our work and some speculative material on a more relaxed form of upcrossings. Section \ref{ergodic_people} is a general discussion of algorithmic randomness intended for ergodic theorists.

\section{Notation and Diagrams}\label{notation}

We will build computable transformations $\widehat T:2^\omega\rightarrow 2^\omega$ using computable functions $T:2^{<\omega}\rightarrow 2^{<\omega}$ such that (1) $\sigma\subseteq\tau$ implies $T(\sigma)\subseteq T(\tau)$ and (2) $\widehat T(x)=\lim_{n\rightarrow\infty}T(x\upharpoonright n)$ is defined and infinite for all $x\in 2^\omega$ outside a computable $\mathcal{G}_\delta$ set with measure $0$.

We will approximate such a $\widehat T$ by \emph{partial transformations}:
\begin{definition}
A \emph{partial transformation} is a computable function $T:2^{<\omega}\rightarrow 2^{<\omega}$ such that if $\sigma\subseteq\tau$ and $T(\tau)$ is defined then $T(\sigma)$ is defined and $T(\sigma)\subseteq T(\tau)$.  We write $T\subseteq T'$ if for all $\sigma$, $T(\sigma)\subseteq T'(\sigma)$.
\end{definition}

We will be exclusively interested in partial transformations which are described finitely in a very specific way:
\begin{definition}
A partial transformation $T$ is \emph{proper} if there are finite sets $T_-,T_+$ such that:
\begin{itemize}
\item $T_-\cup T_+$ is prefix-free,
\item $\cup_{\sigma\in T_-\cup T_+}[\sigma]=2^\omega$,
\item If there is a $\tau\sqsubseteq\sigma$ such that $\tau\in T_-$ then $T(\sigma)=T(\tau)$,
\item If $\sigma=\tau^\frown\rho$ with $\tau\in T_+$ then $T(\sigma)=T(\tau)^\frown\rho$,
 \item If $\sigma\in T_-$ then $|T(\sigma)|<|\sigma|$,
  \item If $\sigma\in T_+$ then $|T(\sigma)|=|\sigma|$,
  \item If $\sigma\in T_+$ and $\sigma\neq\tau\in T_+\cup T_-$ then $T(\tau)\not\supseteq T(\sigma)$.
%\item If $\sigma\in T_+$ and $T(\sigma')\sqsupseteq T(\sigma)$ then $\sigma'\sqsupseteq\sigma$.
\end{itemize}
We say $\sigma$ is \emph{determined} in $T$ if for some such $T_-,T_+$, some initial segment of $\sigma$ belongs to $T_-\cup T_+$.

%We define the \emph{fill} of $\sigma\in T$ by:
%\begin{itemize}
%\item If there is a $\tau\in T_+$ such that $T(\tau)\sqsubseteq\sigma$ then $F_T(\sigma)=\mu([\sigma])$,
%\item Otherwise $F_T(\sigma)=\sum_{\tau\in T_-\cup T_+, T(\tau)\sqsupseteq\sigma}\mu([\tau])$.
%\end{itemize}
%$T$ is \emph{measure-preserving} if for every $\sigma$, $F_T(\sigma)\leq\mu([\sigma])$.

%We say $\sigma$ is \emph{determined in $T$} if $|\sigma|\geq k$.
\end{definition}
%Note that $R,S$, and $T\upharpoonright 2^{\leq k}$ are sufficient to determine a proper transformation.

The roles of $T_-$ and $T_+$ will be clearer when we introduce a diagrammatic notion for describing transformations.  For now, note that a proper transformation is defined by the finite sets $T_-$ and $T_+$ together with the finitely many values of $T$ on these sets.

Throughout this paper, $T$ is always assumed to be proper and measure preserving.

%\begin{lemma}
%If $|\sigma|\geq ht(T)$, $T(\sigma)\subseteq\tau$, and $|\sigma|=|\tau|$, then for any $\upsilon$, %$T(\sigma{}^\frown\upsilon)\subseteq \tau{}^\frown\upsilon$.
%\label{determination}
%\end{lemma}

We will use the method of cutting and stacking, which was introduced by Chacon to produce dynamical systems with specific combinatorial properties \cite{chacon:MR0247028,chacon:MR0207954}\footnote{Actually, according to \cite{friedman:MR1140275}, the method was first used several decades earlier by von Neumann and Kakutani, but not published until later \cite{kakutani:MR0396911}.}.  One tries to construct a dynamical system, usually on the real interval $[0,1]$, by specifying the transformation in stages.  At a given stage, the interval has been ``cut'' into a finite number of components, some of which have been ``stacked'' into ``towers'' or ``ladders.''  A tower is read upwards, so the interval on the bottom level is mapped by the transformation to the level above, and from that level to the level above that.  On the top level of a tower, the transformation is not yet defined.  To produce the next stage, the towers are cut into smaller towers and further stacked.  By manipulating the order in which the components are stacked, specific properties of the transformation can be enforced.  This method has been extensively used in ergodic theory and probability theory to construct examples with specific properties (some overviews of the area are \cite{shields:MR1134300,friedman:MR1140275,katok:MR2186251}).

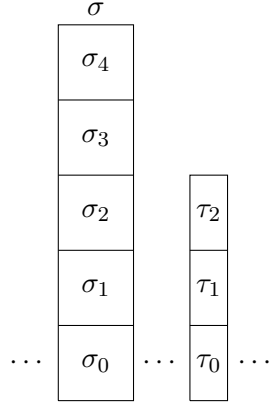
\begin{figure}
\begin{center}
\begin{tikzpicture}
\node[left] at (0,0.5) {$\cdots$};
\node at (1.375,0.5) {$\cdots$};
\node[right] at (2.25,0.5) {$\cdots$};
\draw (0,0) -- (1,0) -- (1,5) -- (0,5) -- (0,0);
\draw (0,1) -- (1,1);
\draw (0,2) -- (1,2);
\draw (0,3) -- (1,3);
\draw (0,4) -- (1,4);
\node at (0.5,0.5) {$\sigma_0$};
\node at (0.5,1.5) {$\sigma_1$};
\node at (0.5,2.5) {$\sigma_2$};
\node at (0.5,3.5) {$\sigma_3$};
\node at (0.5,4.5) {$\sigma_4$};
\node[above] at (0.5,5) {$\sigma$};
\draw (1.75,0) -- (2.25,0) -- (2.25,3) -- (1.75,3) -- (1.75,0);
\draw (1.75,1) -- (2.25,1); 
\draw (1.75,2) -- (2.25,2); 
\node at (2,0.5) {$\tau_0$};
\node at (2,1.5) {$\tau_1$};
\node at (2,2.5) {$\tau_2$};
\end{tikzpicture}
\end{center}
\caption{A typical diagram}
\label{demo_figure}
\end{figure}

A typical diagram is shown in Figure \ref{demo_figure}.  This figure represents that $|\sigma|<|\sigma_0|=|\sigma_1|=|\sigma_2|=|\sigma_3|=|\sigma_4|$ and that for all $\upsilon$, $T(\sigma_i{}^\frown\upsilon)=\sigma_{i+1}{}^\frown\upsilon$ for $i<4$, $T(\sigma_4{}^\frown\upsilon)=\sigma$, and similarly $T(\tau_i{}^\frown\upsilon)=T(\tau_{i+1}{}^\frown\upsilon)$ for $i<2$ while $T(\tau_2)=\langle\rangle$.  Although it is not essential to interpret the diagrams, we will try to be consistent about the scale of blocks; in Figure \ref{demo_figure}, the relative width of the blocks suggests that $|\tau_i|=|\sigma_i|+1$---that is, $\mu([\tau_i])=\mu([\sigma_i])/2$; the height of a block does not represent anything, so we draw each blocks with the same height.  The only relevant dimensions are the widths of the blocks and the numbers of blocks in the towers.  In fact, proper partial transformations can always be represented by diagrams where all blocks have the same width, but it is useful to consider intermediate diagrams where blocks vary in width.

In general, a block represents a subset of $2^\omega$ of the form $[\sigma]$ for some sequence $\sigma$; by placing the block corresponding to $[\sigma]$ on top of the block corresponding to $[\tau]$, we are indicating that $\tau\in T_+$ and $T(\tau)=\sigma$---that is, that in the transformation we construct extending $T$, $T([\tau])=[\sigma]$.  (We must, therefore, have $|\sigma|=|\tau|$.)  By placing some sequence $\sigma'$ with $|\sigma'|<|\sigma|$ on top of the block corresponding to $[\sigma]$, we are indicating that $\sigma\in T_-$ and $T(\sigma)=\sigma'$---that is, in the transformation we construct extending $T$, $T([\sigma])\subseteq[\sigma']$.
%we are indicating partial information about how $T$ acts on $[\sigma]$ (namely, that $T([\sigma])\subseteq[\sigma']$), but leaving open the possibility that $T$ will map subsets of $[\sigma]$ to subsets of $[\sigma']$ in a complicated way

The roles of $T_-$ and $T_+$ in the specification of a proper transformation are now clearer: the elements of $T_-\cup T_+$ are the particular blocks labeled in a given diagram; the elements $\tau\in T_+$ are those blocks which have another block on top, and therefore $T$ is completely defined on every element of $[\tau]$.  The elements $\tau\in T_-$ are topmost blocks of some tower, for which we have (at most) partial information about $T$ on $[\tau]$.

We will only be concerned with partial transformations satisfying two additional properties.
\begin{definition}
An \emph{open loop} in a partial transformation $T$ is a sequence $\sigma_0,\ldots,\sigma_n$ such that:
\begin{itemize}
  \item $|\sigma_0|=|\sigma_1|=\cdots=|\sigma_n|$,
  \item $T(\sigma_i)=\sigma_{i+1}$ for $i<n$,
  \item $T(\sigma_n)\sqsubset \sigma_0$.
\end{itemize}

%We say an open loop is maximal if there is no $\sigma'$ with $|\sigma'|=|\sigma_0|$ and $T(\sigma')=\sigma_0$.  
The \emph{width} of a loop is the value $2^{-|\sigma_i|}$, and the \emph{volume} of a loop is $n2^{-|\sigma_i|}$.

We say $T$ is \emph{partitioned into open loops} if for every $\sigma$ there is an open loop $\sigma_0,\ldots,\sigma_n$ in $T$ with $\sigma=\sigma_i$ for some $i$.  (In a proper transformation such a loop must be unique.)  In such a transformation we write $\mathcal{L}_T(\sigma)$ for the loop $\sigma_0,\ldots,\sigma_n$ such that for some $i$, $\sigma=\sigma_i$.  We write $\iota_T(\sigma)$ for this value of $i$.%\footnote{The property of being partitioned into open loops is quite powerful.  For instance, with arbitrary partial transformations, we might have to do extensive work to ensure that the transformation is measure-preserving.  Instead the partition into open loops, combined with other elements of our construction, will ensure this easily.}

We say $\tau$ is \emph{blocked} if there is any $\sigma$ such that $T(\sigma)\sqsupset\tau$.  Otherwise we say $\tau$ is \emph{unblocked}.

%We say that $S\subseteq 2^{\omega}$ is \emph{blocked} in $T$ if there exist $\sigma^0,\ldots,\sigma^k$ such that $S\subseteq\cup_{j\leq k}[\sigma^j]$ and for every $j\leq k$, $\iota_T(\sigma^j)\neq 0$.  We say $\tau$ is blocked if $[\tau]$ is.  Otherwise we say $\tau$ is \emph{unblocked} in $T$.
\end{definition}
(We are interested in open loops to preclude the possibility that $T(\sigma_n)=\sigma_0$, since we are not interested in---indeed, will not allow the existence of---``closed'' loops.)  Diagrammatically, the requirement that $T$ be partitioned into open loops is represented by requiring that any sequence written above a tower of blocks is a subsequence of the sequence at the bottom of that tower.  (For instance, in Figure \ref{demo_figure}, we require that $\sigma\sqsubset\sigma_0$.)

% Being blocked represents that $\tau$ is a middle level in the corresponding tower in the diagram, or, if $[\tau]$ has been divided into several distinct blocks, that all these blocks are middle levels in their diagram.  To see the significance of this notion, consider the situation in Figure \ref{blocking_figure}.  We cannot define $T(\rho)\sqsupseteq\sigma$ for some new element $\rho$.  In some sense being unblocked is dual to being determined---determined elements are ones it is safe to add to the domain of $T$, and unblocked elements may be added to the range.

%  \begin{figure}
%    \begin{tikzpicture}
%  \draw (0,0) -- (1.2,0) -- (1.2,2) -- (0,2) -- (0,0);
%  \draw (0,1) -- (1.2,1);
%  \draw (1.75,0) -- (2.95,0) -- (2.95,2) -- (1.75,2) -- (1.75,0);
% \draw (1.75,1)--(2.95,1);
%  \node at (0.6,0.5) {$\tau$};
%  \node at (0.6,1.5) {$\sigma^\frown\langle 0\rangle$};
% \node at (2.35,0.5) {$\upsilon$};
%  \node at (2.35,1.5) {$\sigma^\frown\langle 1\rangle$};
%    \end{tikzpicture}
%    \caption{Blocking}
%    \label{blocking_figure}
%  \end{figure}

\begin{definition}
An \emph{escape sequence} for $\sigma$ in $T$ is a sequence $\sigma_0,\ldots,\sigma_n$ such that:
\begin{itemize}
\item $\sigma_{0}=\sigma$,
\item $|\sigma_1|=|\sigma_2| = \ldots = |\sigma_n|$,
\item For all $0\leq i<n$, $\sigma_{i+1}\sqsupseteq T(\sigma_i)$,
\item If $\sigma_{i+1}$ is blocked then $\sigma_{i+1}=T(\sigma_i)$,
\item $T(\sigma_n)=\langle\rangle$,
\item All $\sigma_i$ are determined.
\end{itemize}

We say $T$ is \emph{escapable} if for every determined $\sigma$ with $|T(\sigma)|<|\sigma|$, there is an escape sequence for $\sigma$.

An escape sequence is \emph{reduced} if the following two conditions hold: (1) $\sigma_i\supseteq T(\sigma_j)$ implies that %$i\leq j+1$ and if $i<n$ implies $T(\sigma_i)\neq\langle\rangle$.
either $i\leq j+1$ or $\sigma_i$ is blocked, and (2) if $i<n$, then $T(\sigma_i)\neq\langle\rangle$.
\end{definition}
Escapability preserves the option of extending $T$ in such a way that we can eventually map $[\sigma_0]$ to anything not already in the image of another sequence (although it may require many applications of $T$).

\begin{lemma}
Every escape sequence for $\sigma$ contains a reduced subsequence for $\sigma$.
\label{escape_extend}
\end{lemma}
\begin{proof}
%  By induction on the length of the sequence.  It suffices to show that if $\sigma_0,\ldots,\sigma_n$ is a non-reduced escape sequence then there is a proper subsequence which is also an escape sequence for $\sigma_0$.  But if $\sigma_0,\ldots,\sigma_n$ is not reduced, for some $i>j+1$, $\sigma_i\supseteq T(\sigma_j)$, and therefore $\sigma_0,\ldots,\sigma_j,\sigma_i,\ldots,\sigma_n$ is also an escape sequence.

  We proceed by induction on the length of the sequence.  It suffices to show that if $\sigma_0,\ldots,\sigma_n$ is a nonreduced escape sequence then there is a proper subsequence which is also an escape sequence for $\sigma_0$.  If for some $i>j+1$, $\sigma_i\supseteq T(\sigma_j)$ with $\sigma_i$ unblocked, then $\sigma_0,\ldots,\sigma_j,\sigma_i,\ldots,\sigma_n$ is also an escape sequence.  If for some $i<n$, $T(\sigma_i)=\langle\rangle$ then $\sigma_0,\ldots,\sigma_i$ is also an escape sequence.
\end{proof}

\begin{lemma}
  If $\sigma_0,\ldots,\sigma_n$ is an escape sequence for $\sigma_0$ in $T$ then for every $\rho$, $\sigma_0,\sigma_1{}^\frown\rho,\ldots,\sigma_n{}^\frown\rho$ is also an escape sequence.
\label{thin_escape_sequence}
\end{lemma}

%UNOUTCOMMENT PROOF?

%\begin{proof}
%We choose a sufficiently long $\rho$ and set $\sigma'_i=\sigma_i{}^\frown\rho$ for $i>0$ and $\sigma'_0=\sigma_0$.

%We proceed as follows: we let $\sigma'_0=\sigma_0$.  Given $\sigma_0,\ldots,\sigma_i$ with $\sigma_{i+1}$ unblocked, we choose a $\rho$ with $|\rho|=n-|\sigma_1|$ such that $\sigma_{i+1}{}^\frown\rho$ is unblocked (this must exist, since otherwise $\sigma_{i+1}$ would be blocked).  Let $r$ be least such that $\sigma_{i+2+r}$ is unblocked, and for $j\in[i+1,i+2+r-1]$, set $\sigma'_j=\sigma_j{}^\frown\rho$.

%When $\sigma_{i+1}$ is blocked, we have $T(\sigma_i)=\sigma_{i+1}$, so also $T(\sigma_i{}^\frown\rho)=\sigma_{i+1}{}^\frown\rho$.  When $\sigma_{i+1}$ is not blocked, $T(\sigma_i)\sqsubset\sigma_{i+1}$, and since $\sigma_i$ is determined, $T(\sigma'_i)=T(\sigma_i)\sqsubset\sigma_{i+1}{}^\frown\rho$.
%\end{proof}

\begin{lemma}
  Suppose $\sigma_0,\ldots,\sigma_k$ is an open loop in $T$ consisting of determined elements such that for all $\upsilon$, $\sigma_0\sqsubseteq T(\upsilon)$ implies $\sigma_0=T(\upsilon)$ and $\tau_0,\ldots,\tau_n$ is a reduced escape sequence for $\tau_0$. Then one of the following occurs:
  \begin{itemize}
  \item $\tau_0\sqsupseteq\sigma_k$ and for $j>0$, $\tau_j\not\in\cup_{i\leq k}[\sigma_i]$,
  \item There is a unique $j>0$ such that for all $i\leq k$, $\tau_{j+i}\sqsupseteq\sigma_i$,
  \item For all $j$, $\tau_j\not\in\cup_{i\leq k}[\sigma_i]$.
  \end{itemize}
\label{escape_in_loops}
\end{lemma}
\begin{proof}
  First, suppose some $\tau_j\sqsupseteq\sigma_i$.  If $j=0$ then since $|T(\tau_0)|<|\tau_0|$, we must have $i=k$.  If $j\neq 0$ and $i\neq 0$ then since $\sigma_i$ is blocked, we must have $\tau_{j-1}\sqsupseteq\sigma_{i-1}$.  Since $\tau_0\not\sqsupseteq\sigma_{i-1}$, we can repeat this and conclude that $j>i$ and $\tau_{j-i}\sqsupseteq\sigma_0$.  Furthermore, for each $i'\leq k$, $\tau_{j-i+i'}\sqsupseteq\sigma_{i'}$; we have already shown this for $i'\leq i$, and for $i'>i$ it follows since $T(\tau_j)\sqsupseteq\sigma_{i+1}$, and so on.

So we have shown that if $\tau_j\sqsupseteq\sigma_i$ for some $i,j$ with $j>0$ then we have a complete copy of the loop in our escape sequence.  We now show that if $j<j'$ and $\tau_j\sqsupseteq\sigma_k$ then we cannot have $\tau_{j'}\sqsupseteq\sigma_i$; this shows both the second half of the first case and the uniqueness in the second case.  For suppose we had $\tau_j\sqsupseteq\sigma_k$, $j'>j$, and $\tau_{j'}\sqsupseteq\sigma_i$.  By the previous paragraph, we may assume $i=k$.  But since $\sigma_k$ is determined and $|T(\sigma_k)|<\sigma_k$, we have $T(\tau_j)=T(\sigma_k)=T(\tau_{j'})$.  But then we either have $T(\tau_{j'})=\langle\rangle$ or $T(\tau_j)\sqsubseteq\tau_{j+1}$; in either case the sequence is not reduced, contradicting our assumption.
\end{proof}

% \begin{proof}
% Since $\sigma$ is determined in $T$, we consider two cases.  Either there is a $\sigma'\subseteq\sigma$ such that $T(\sigma{}^\frown\upsilon)=T(\tau)$ for all $\upsilon$, in which case the claim is immediate, or $T(\sigma'{}{}^\frown\rho)=T(\sigma'){}^\frown\rho$ for all $\rho$, in which case the claim is also immediate.
% \end{proof}

\section{Pieces of the Construction}\label{lemmas}
In this section we describe certain modifications of partial transformations and show that they preserve certain essential properties.  First, we lay out the technical properties of the transformations we need to construct.

\begin{definition}
A partial transformation $T$ is \emph{useful} if:
\begin{itemize}
  \item $T$ is proper,
  \item $T$ is partitioned into open loops, and
  \item $T$ is escapable.
\end{itemize}
\end{definition}

It will be helpful to keep track of the following quantity:
\begin{definition}
For any $\sigma$, the \emph{burden} of $\sigma$ in $T$, $bd_T(\sigma)$, is $\sum_{i\leq k}\mu([\sigma])=k2^{-|\sigma|}$ where $k$ is the length of $\mathcal{L}_T(\sigma)$.
\end{definition}

\begin{lemma}[Thinning Loops]
Let $T$ be a useful partial transformation, let $\sigma_0,\ldots,\sigma_k$ be an open loop of determined elements such that if $\tau\in T_-\cup T_+$ then $T(\tau)\not\sqsupset\sigma_0$, and let $\epsilon=2^{-n}$ be smaller than the width of this loop.  Then there is a useful $T'\supseteq T$ such that:
\begin{itemize}
  \item There is a loop $\tau_0,\ldots,\tau_{k'}$ in $T'$ of width $\epsilon$ such that $\cup_{j\leq k'}[\tau_i]=\cup_{i\leq k}[\sigma_i]$,
  \item If $\tau\not\in\cup_{i\leq k}[\sigma_i]$ then $T'(\tau)=T(\tau)$.
\end{itemize}
Furthermore, the burden of $\sigma$ in $T$ is the same as the burden of $\sigma$ in $T'$ for every $\sigma$.
\label{thin_loop}
\end{lemma}

\begin{figure}
\begin{center}
\begin{tikzpicture}
\node[left] at (0,0.5) {$\cdots$};
\node[right] at (1.5,0.5) {$\cdots$};
\draw (0,0) -- (1.5,0) -- (1.5,2) -- (0,2) -- (0,0);
\draw (0,1) -- (1.5,1);
\node at (0.75,0.5) {$\sigma_0$};
\node at (0.75,1.5) {$\sigma_1$};
\node[above] at (0.75,2) {$\sigma$};
\node[below] at (0.75,0) {Before};

\node[left] at (4,0.5) {$\cdots$};
\node[right] at (5.5,0.5) {$\cdots$};
\draw (4.75,0) -- (4,0) -- (4,2);
\draw[decorate,decoration={snake, amplitude=.3mm,segment length=2mm}] (4,2) -- (4.75,2);
\draw (4.75,2) -- (5.5,2) -- (5.5,0);
\draw[decorate,decoration={snake, amplitude=.3mm,segment length=2mm}] (4.75,0) -- (5.5,0);
\node at (4.75,0.5) {$\sigma_0$};
\node at (4.75,1.5) {$\sigma_1$};
\draw[dashed] (4.75,0) -- (4.75,2);
%\draw[dotted,->] (4.375,2.1) -- (4.375,2.2);
\node[below] at (4.75,0) {During};

\node[left] at (8,0.5) {$\cdots$};
\node[right] at (8.75,0.5) {$\cdots$};
\draw (8,0) -- (8.75,0) -- (8.75,4) -- (8,4) -- (8,0);
\draw (8,1) -- (8.75,1);
\draw[decorate,decoration={snake, amplitude=.3mm,segment length=2mm}] (8,2) -- (8.75,2);
\draw (8,3) -- (8.75,3);
\node at (8.375,0.5) {$\sigma_00$};
\node at (8.375,1.5) {$\sigma_10$};
\node at (8.375,2.5) {$\sigma_01$};
\node at (8.375,3.5) {$\sigma_11$};
\node[above] at (8.375,4) {$\sigma$};
\node[below] at (8.375,0) {After};
\end{tikzpicture}
\end{center}
\caption{Thinning Loops, Lemma \ref{thin_loop}}
\label{thin_loop_figure}
\end{figure}
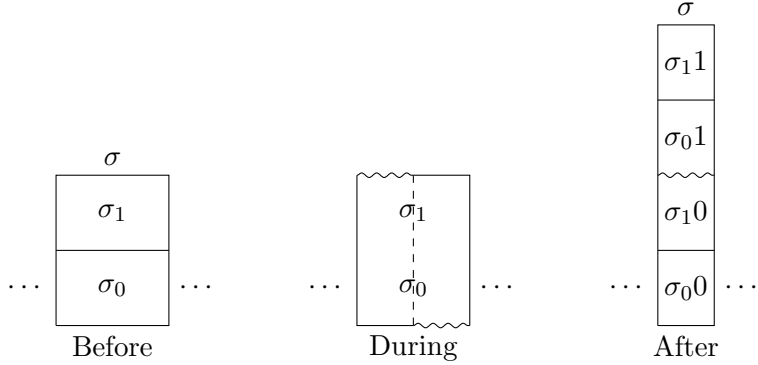

\begin{proof}
Figure \ref{thin_loop_figure} illustrates this lemma.  Formally, let the width of $\sigma_0,\ldots,\sigma_k$ be $2^{-m}$ with $m\leq n$.  Define $T'\supseteq T$ by:
\begin{itemize}
  \item If $\tau=\sigma_k{}^\frown\upsilon{}^\frown\rho$ where $|\upsilon|=n-m$ and $\upsilon$ is not all $1$'s then $T'(\tau)=\sigma_0{}^\frown(\upsilon+1){}^\frown\rho$,
  \item Otherwise $T'(\tau)=T(\tau)$.
\end{itemize}
Note that $T\subseteq T'$ implies that $T'(\sigma_i{}^\frown\rho)=\sigma_{i+1}{}^\frown\rho$ for $i<k$.

Propriety and the fact that $T'$ is partitioned into loops are trivial.  To see escapability, consider some $\tau$ determined such that $|T'(\tau)|<|\tau|$ and fix a reduced escape sequence $\tau_0,\ldots,\tau_r$ in $T$.  If $\tau\in\cup[\sigma_i]$ then by Lemma \ref{escape_in_loops} we have $\tau\sqsupseteq\sigma_k$ and no other element of the escape sequence belongs to $\cup[\sigma_i]$, and therefore $\tau_0,\ldots,\tau_r$ is an escape sequence in $T'$ as well.

If $\tau\not\in\cup[\sigma_i]$ but there is a $j>0$ such that $\tau_{j+i}\sqsupseteq\sigma_i$ for $j\leq k$ then by Lemma \ref{thin_escape_sequence} we may assume that for $j>0$, $|\tau_j|\geq n$.  Then since both $T(\tau_{j-1})\sqsubseteq \tau_j$ and $\sigma_0\sqsubseteq\tau_j$, we must have $T(\tau_{j-1})\sqsubseteq\sigma_0$, and therefore for any $\rho$ of suitable length, the sequence
\[\tau_0,\ldots,\tau_{j-1},\sigma_0{}^\frown\langle 0,\ldots,0\rangle{}^\frown\rho,\ldots,\sigma_k{}^\frown\langle 1,\ldots,1\rangle^\frown\rho,\tau_{j+k+1},\ldots,\tau_n\]
 is an escape sequence for $\tau$ in $T'$.
\end{proof}

We also need a modified version of the above lemma where instead of wanting $\cup_{j\leq k'}[\tau_i]=\cup_{i\leq k}[\sigma_i]$ we want to have a small amount of the original loop left alone.
\begin{lemma}
Let $T$ be a useful partial transformation, let $\sigma_0,\ldots,\sigma_k$ be an open loop of determined elements such that if $\tau\in T_-\cup T_+$ then $T(\tau)\not\supset\sigma_0$, and let $\epsilon=2^{-n}$ be smaller than the width of this loop.  Then there is a useful $T'\supseteq T$ with the following properties:
\begin{itemize}
  \item There is a loop $\tau_0,\ldots,\tau_{k'}$ in $T'$ of width $\epsilon$ such that $\left(\cup_{j\leq k'}[\tau_i]\right)\setminus\left(\cup_{i\leq k}[\sigma_i]\right)=\epsilon\cdot k$,
  \item If $\tau\not\in\cup_{i\leq k'}[\tau_i]$ then $T'(\tau)=T(\tau)$,
  \item If $\tau\not\in\cup_{i\leq k'}[\tau_i]$ and $\tau$ has an escape sequence in $T$ then $\tau$ has an escape sequence in $T'$ which does not contain any element of $\cup_{i\leq k'}[\tau_i]$.
\end{itemize}
Furthermore the burden of $\sigma$ in $T$ is the same as the burden of $\sigma$ in $T'$ for every $\sigma$.
\label{thin_loop_with_leftover}
\end{lemma}
\begin{proof}
  We proceed exactly as above except that we leave the strip $\sigma_i{}^\frown\langle 1,\ldots,1\rangle$ untouched and leave $T'(\sigma_k{}^\frown\langle 1,\ldots,1,0\rangle)=T(\sigma_k)$.
\end{proof}

\begin{lemma}[Lockstep Escape]
Let $T$ be a useful partial transformation and let $\sigma_0,\ldots,\sigma_k$ be a reduced escape sequence such that $|\sigma_0|=|\sigma_1|+1$.  Suppose that $T_-,T_+$ witness that $T$ is proper, and define
\[T'_+=T_+\cup\{\sigma_0\}\cup\{\sigma_i{}^\frown\langle 0\rangle\mid i>0\}.\]
Setting $T'(\sigma_0)=\sigma_1{}^\frown\langle 0\rangle$, and for $i>0$, $T'(\sigma_i{}^\frown\langle 0\rangle)=\sigma_{i+1}{}^\frown\langle 0\rangle$ fully specifies a partial transformation $T'\supseteq T$.  Then $T'$ is useful and if $\sigma\not\in\cup_{i\leq k}[\sigma_i]\cup\cup_{j\leq k'}[\tau_j]$, where $\tau_0,\ldots,\tau_{k'}$ is the open loop containing $\sigma_0$, then the burden of $\sigma$ is unchanged.
\label{escaping}
\end{lemma}
\begin{proof}
  To see that $T'$ is proper, we note that we have specified $T'_+$ and (implicitly) $T'_-$, and we need only check that if $\sigma\in T'_+$ and $\sigma\neq \tau\in T'_+\cup T'_-$ then $T'(\tau)\not\supseteq T'(\sigma)$.  Clearly we need only check this for $T(\sigma)=\sigma_i{}^\frown\langle 0\rangle$.  Since the escape sequence was reduced, we cannot have $\sigma_i=\sigma_j$ for $i\neq j$, so we can restrict our attention to the $\tau$ such that $T'(\tau)=T(\tau)$.  If $\sigma_i$ was not blocked in $T$ then there is no such $\tau$, and if $\sigma_i$ was blocked in $T$ then already $T(\sigma_{i-1}{}^\frown\langle 0\rangle)=\sigma_i{}^\frown\langle 0\rangle$, and the claim follows since $T$ was proper.

It is easy to see that $T'$ remains partitioned into open loops.

Finally we check that $T'$ is escapable.  Let $\tau$ determined be given with $|T'(\tau)|<|\tau|$.  Then the same was true in $T$, so $\tau$ had an escape sequence $\tau_0,\ldots,\tau_r$ in $T$.  We may assume $|\tau_1|\geq|\sigma_1|$.  There are a few potential obstacles we need to deal with.  First, it could be that for some $i$, $\tau_i\sqsupseteq\sigma_0$.  Letting $\tau_i=\sigma_0{}^\frown\rho$, we must have $\tau_0,\ldots,\sigma_0{}^\frown\rho,\sigma_1{}^\frown\langle 0\rangle^\frown\rho,\ldots,\sigma_k{}^\frown\langle 0\rangle^\frown\rho$ is also an escape sequence for $\tau_0$.

Otherwise, there could be some $i$ and $j>0$ such that $\tau_i\sqsupseteq\sigma_j$.  If $\tau_i=\sigma_j$ then $\tau_0,\tau_1{}^\frown\langle 1\rangle,\ldots,\tau_r{}^\frown\langle 1\rangle$ is also an escape sequence in $T$ and remains one in $T'$.  If $\tau_i=\sigma_j{}^\frown\langle 1\rangle^\frown\rho$, the same escape sequence works in $T'$.  Otherwise, replacing each such $\tau_i$ with $\sigma_j{}^\frown\langle 1\rangle^\frown\rho$ gives a new escape sequence in $T$ which remains one in $T'$.
\end{proof}

%\begin{lemma}[Refining Escape Sequences]
%Let $T$ be a proper partial transformation, let $\sigma_0,\ldots,\sigma_k$ be an escape sequence, and let $\tau_0\sqsupseteq\sigma_0$.  Then there exist $\tau_1,\ldots,\tau_k$ such that $\tau_0,\ldots,\tau_k$ is an escape sequence.
%\label{refine_escape}
%\end{lemma}
%\begin{proof}
%  If $\tau_0=\sigma_0{}^\frown\rho$, we simply take $\tau_i=\sigma_i{}^\frown\rho$ for all $i$.
%\end{proof}

% Combing the pieces above gives the main tool of our construction:
% \begin{lemma}
% Let $\sigma$ be determined in a useful transformation $T$ and let $\delta>0$.  There is an $n\geq|\sigma|$ such that for any $n'\geq n$, there is a useful transformation $T'\supseteq T$ and an open loop in $T'$ of width $1/2^{n'}$, $\tau_0,\ldots,\tau_k$, such that:
% \begin{itemize}
%   \item $[\sigma]\subseteq\cup_{i\leq k}[\tau_i]$,
%   \item $T'(\tau_k)=\langle\rangle$,
%   \item $\sum_{i\leq k}\mu([\tau_i])\leq bd_T(\sigma)+\delta$,
%   \item If $\upsilon\not\in\cup_{i\leq k}[\tau_i]$ then $bd_{T'}(\upsilon)=bd_T(\upsilon)$ and $T'(\upsilon)=T(\upsilon)$.
% \end{itemize}
% \label{combine}
% \end{lemma}

The main building block of our construction will combine these two steps as illustrated in Figure \ref{combine_figure}.  The next section is devoted to explaining this process.

\begin{figure}
\begin{center}
\begin{tikzpicture}[scale=0.75]
    \tikzstyle{every node}=[font=\small]
\node[left] at (0,0.25) {$\cdots$};
\node[right] at (3,0.25) {$\cdots$};
\draw (0,0) -- (1,0) -- (1,1.5) -- (0,1.5) -- (0,0);
\draw (0,.5) -- (1,.5);
\draw (0,1) -- (1,1);
\draw (1,0) -- (2,0) -- (2,.5) -- (1,.5) -- (1,0);
\draw (2,0) -- (3,0) -- (3,1) -- (2,1) -- (2,0);
\draw (2,.5) -- (3,.5);
\node[above] at (0.5,1.5) {$\langle\rangle$};
\node[below] at (1.5,0) {Before};

\node[left] at (5,0.25) {$\cdots$};
\node[right] at (8,0.25) {$\cdots$};

\draw (5.5,0) -- (5,0) -- (5,1.5) -- (5.5,1.5);
\draw[fill=gray!50] (5.5,0)
decorate [decoration={zigzag}] {-- (5.75,0)}
[dashed] -- (5.75,1.5)
[-] -- (5.5,1.5)
[dashed] -- (5.5,0);
\draw  (5.5,1.5)-- (5.75,1.5);
\draw (5.75,0) -- (6,0) -- (6,1.5) --(5.75,1.5);
\draw (5,.5) -- (6,.5);
\draw (5,1) -- (6,1);
\draw (6.25,0) -- (6,0) -- (6,.5) -- (6.25,.5);
\draw[pattern=north west lines] (6.25,0) 
decorate [decoration=zigzag] {-- (6.5,0)}
-- (6.5,.5)
decorate [decoration={zigzag,mirror}] {-- (6.25,.5)}
-- (6.25,0);
\draw (6.5,0) -- (7,0) -- (7,.5) -- (6.5,.5);
\draw (7,0) -- (7.25,0)
decorate [decoration={zigzag}] {-- (7.5,0)}
decorate [decoration={zigzag}] {-- (7.75,0)}
decorate [decoration={zigzag}] {-- (8,0)}
-- (8,1)
decorate [decoration={zigzag,mirror}] {-- (7.75,1)}
decorate [decoration={zigzag,mirror}] {-- (7.5,1)}
decorate [decoration={zigzag,mirror}] {-- (7.25,1)}
decorate [decoration={zigzag,mirror}] {-- (7,1)}
--(7,0);
 \draw[dashed] (7.25,0) -- (7.25,1);
 \draw[dashed] (7.5,0) -- (7.5,1);
 \draw[dashed] (7.75,0) -- (7.75,1);
\draw (7,0.5)--(8,0.5);
\node[below] at (6.5,0) {During};

\node[left] at (11,0.25) {$\cdots$};
\node[right] at (13.25,0.5) {$\cdots$};
\draw (11,0) -- (11.5,0) -- (11.5,1.5) -- (11,1.5) -- (11,0);
\draw (11.75,0) -- (12,0) -- (12,1.5) -- (11.75,1.5) -- (11.75,0);
\draw (11,.5) -- (11.5,.5);
\draw (11,1) -- (11.5,1);
\draw (11.75,.5) -- (12,.5);
\draw (11.75,1) -- (12,1);
\draw (12,0) -- (12.25,0) -- (12.25,.5) -- (12,.5) -- (12,0);
\draw (12.5,0) -- (13,0) -- (13,.5) -- (12.5,.5) -- (12.5,0);
\draw (13,0) -- (13.25,0) -- (13.25,1) 
decorate [decoration={zigzag,mirror}] {-- (13,1)}
-- (13,0);
\draw (13,.5) -- (13.25,.5);
\node at (13.125,1.5) {$\vdots$};
\draw (13,1.7)
decorate [decoration={zigzag}] {-- (13.25,1.7)}
-- (13.25,2.7);
\draw (13,2.7)--(13,1.7);
\draw (13,2.2) -- (13.25,2.2);
\draw[draw=white,pattern=north west lines] (13,2.7)
decorate [decoration={zigzag}] { -- (13.25,2.7)}
 -- (13.25,3.2) 
decorate [decoration={zigzag,mirror}] {-- (13,3.2)}
 -- (13,2.7);
\draw (13,3.2) -- (13,2.7)
decorate [decoration={zigzag}] {-- (13.25,2.7)}
-- (13.25,3.2) ;
\draw[fill=gray!50] (13,3.2) 
decorate [decoration={zigzag}] {-- (13.25,3.2)}
-- (13.25,4.7)-- (13,4.7) -- (13,3.2);
 \draw (13,3.7) -- (13.25,3.7);
 \draw (13,4.2) -- (13.25,4.2);
\node[below] at (12.5,0) {After};

\end{tikzpicture}
\end{center}
\caption{}
\label{combine_figure}
\end{figure}

% \begin{proof}
%   See Figure \ref{combine_figure}.  Letting $\sigma_0,\ldots,\sigma_k$ be the open loop containing $\sigma$ and $\sigma_k=\upsilon_0,\ldots,\upsilon_r$ be an escape sequence for $\sigma_k$, we apply Lemma \ref{thin_loop} with $\epsilon\leq\delta/r$.  We then refine the escape sequence and apply lockstep escape with this sequence to ensure that $T'(\tau_k)=\langle\rangle$.  The resulting loop consists of our original loop, with measure $bd_T(\sigma)$, plus the measure of the escape sequence, which is $r\epsilon\leq\delta$.
% \end{proof}

\section{The Main Construction}\label{main}
Our main tool for causing the Birkhoff ergodic theorem to fail at a point is the notion of an upcrossing.
\begin{definition}
Given a measurable, measure-preserving, invertible $T:2^\omega\rightarrow 2^\omega$, a point $x\in 2^\omega$, a measurable $f$, and rationals $\alpha<\beta$, an \emph{upcrossing sequence} for $\alpha,\beta$ is a sequence
\[0\leq u_1<v_1<u_2<v_2<\cdots<u_N<v_N\]
such that for all $i\leq N$,
\[\frac{1}{u_i+1}\sum_{j=0}^{u_i}f(T^jx)<\alpha,\ \frac{1}{v_i+1}\sum_{j=0}^{v_i}f(T^jx)>\beta.\]

$\tau(x,f,\alpha,\beta)$ is the supremum of the lengths of upcrossing sequences for $\alpha,\beta$.
\end{definition}
By definition, Birkhoff's ergodic theorem fails at $x$ exactly if $\tau(x,f,\alpha,\beta)=\infty$ for some $\alpha<\beta$.  Our plan is to look at an \ml test $\langle V_j\rangle$ and, as sequences $\sigma$ are enumerated into an appropriate $V_j$, ensure that the lower bound on $\tau(x,f,1/3,1/2)$ increases for each $x\in[\sigma]$.

\begin{theorem}\label{main_thm}
Suppose $x\in 2^\omega$ is not \ml random.  Then there is a computable set $A$ and a computable transformation $T:2^\omega\rightarrow 2^\omega$ such that $x$ is not typical with respect to the ergodic theorem.
\end{theorem}

\begin{proof}
Let $\langle V_j\rangle$ be a \ml test witnessing that $x$ is not \ml random, so $x\in\cap_j V_j$.  We will construct an increasing sequence of useful partial transformations $T_0\subseteq T_1\subseteq T_2\subseteq\cdots$ so that $T=\cup_n T_n$ will be the desired transformation.  We will maintain, at each stage, a computable partition of $2^\omega$ into clopen components $W^n$, and for each value $k$, we will maintain components $A^n_k$ and $B^n_k$ with the requirement that it is always possible to extend $T_n$ such that it maps each component to itself.  (Since the partition components are all clopen sets, we will also treat these sets as a partition of $2^{\geq m}$ for $m$ sufficiently large.)

$W^n$ represents the ``work area''; initially $W^0$ will be a portion of $2^\omega$ known to contain $x$.  In later stages, $W^n$ will grow to include parts of the $A^n_k$ and $B^n_k$ which have been used.  $A=\cup_k A^0_k$ will be the set which will demonstrate the failure of the ergodic theorem for $x$.  Our strategy will then be that when we discover elements in $V_n$ for appropriate $n$, we will arrange for the transformation to eventually map those elements through $A$ for a long time, ensuring that the average membership in $A$ reaches $1/2$.  We will then have the transformation map those elements through $B=\cup_k B^0_k$ for a long time to bring the average down to $1/3$.  We will do this to each element of $\cap_j V_j$ infinitely many times, ensuring that elements in this intersection are not typical.\footnote{It is not possible to ensure that every element of the set $V_j$ receives $j$ upcrossings, since this would imply that the theorem holds for every $x$ which failed to be even Demuth random, which would contradict V'yugin's theorem.}  $A^0_k$ is the section of $A$ reserved for making the average large for the $(k+1)^{st}$ time, and $B^0_k$ is the section reserved for making the average small again after.  $A^n_k$ and $B^n_k$ represent the portions still available at stage $n$ after some parts have been used. % We will maintain that all loops in $A^n_k, B^n_k$ have length $1$.
  We also keep track of constants $a^n_k<\mu(A^n_k)$ and $b^n_k<\mu(B^n_k)$, which represent how much of $A^n_k$ and $B^n_k$ have already been committed but not yet used.

Finally we have a partition $W^n=\cup_k W^n_k$, where elements of $W^n_k$ are those which are already guaranteed in stage $n$ to have $k$ upcrossings, and a function $\rho^n:T_+\cup T_-\rightarrow\mathbb{N}$, where $V_{\rho^n(\sigma)}$ is the element of our test set we will be watching to discover which elements of $[\sigma]$ require a new upcrossing.

Initially, we assume without loss of generality that $x$ belongs to some clopen set $W^0$ with $\mu(W^0)<1$. For instance, we can suppose we know the first bit of $x$ and let $W^0$ equal $[0]$ or $[1]$ as appropriate. Then, from the remaining measure, we take $A^0_k$ and $B^0_k$ so that $\mu(B^0_k)=2\mu(A^0_k)$ for all $k$.  We set $W^0_0=W^0$ and take $T_0$ to be the trivial transformation of height $0$ (i.e., $T_0(\sigma)=\langle\rangle$ for all $\sigma$).  Choose $j$ large enough that $\mu(V_j)< \mu(A^0_0)$ and set $\rho^0(\langle\rangle)=j$.  Set $a^0_0=\mu(V_j)$, $b^0_0=2\mu(V_j)$, and for $k>0$, $a^0_k=b^0_k=0$.

We require that all determined members of $A^n_k$ and $B^n_k$ are unblocked.  Finally, we will maintain a stronger form of escapability: we require that if $\sigma$ is determined and $|T_n(\sigma)|<|\sigma|$, then $\sigma$ belongs to either $W^n$ or $A^n_k$ or $B^n_k$ for some $k$, and we require that $\sigma$ has an escape sequence contained entirely in the same component.

Now we proceed by stages. At each even stage $n$, we take steps to ensure that $T$ is defined almost everywhere. Given $T_n$, defined by $T_+,T_-$, we define $T'_+=T_+$ and define $T'_-$ by
\[T'_-=\cup_{\sigma\in T_-}\{\sigma^\frown\langle 0\rangle,\sigma^\frown\langle 1\rangle\}.\]
We then set $T_{n+1}(\sigma^\frown\langle 0\rangle)=T_n(\sigma)$.  Let $\tau$ be the initial element of the open loop containing $\sigma$ in $T$, so $T(\sigma)\sqsubset\tau$.  In particular, there is a $b\in\{0,1\}$ such that $T(\sigma)^\frown\langle b\rangle\sqsubseteq\tau$, and we set $T_{n+1}(\sigma^\frown\langle 1\rangle)=T(\sigma)^\frown\langle b\rangle$.  Note that this process ensures that as long as $x\in 2^{\omega}$ does not end in cofinitely many $0$'s, $T(x)$ will be defined.

We now consider the real work. At each odd stage $n$, we take steps to ensure that the ergodic theorem does not hold for any element of $\cap_j V_j$. We may assume that exactly one $\tau$ is enumerated into $\cup_j V_j$ at this stage, and that if $\tau$ is enumerated into $V_j$ at this stage then for each $i<j$, there is a $\tau'\sqsubseteq\tau$ which was enumerated into $V_i$ at some previous stage.  We first assume $\tau$ is determined.  If $\rho^n(\tau)\neq j$ where $\tau$ was enumerated into $V_j$, we do nothing, so we will assume that $\rho^n(\tau)=j$.  We have $\tau\in W^n_{k-1}$, and we have ensured inductively that $bd_{T_n}(\tau)\leq a^n_k<\mu(A^n_k)$ and $2bd_{T_n}(\tau)\leq b^n_k<\mu(B^n_k)$.

In Figure \ref{one_step_figure}, we illustrate the way we intend to arrange $T'$.  We must ensure that every point in $[\tau]$, a section of fixed total measure, receives a new upcrossing.  We must do so while ensuring that the total measure of the portions of $A^n_k$ used is strictly less than $bd_{T_n}(\tau)+(\mu(A^n_k)-a^n_k)$ and the total measure of the portions of $B^n_k$ used is strictly less than $2bd_{T_n}(\tau)+(\mu(B^n_k)-b^n_k)$.  Finally, the escape sequences will all have a fixed height, which we cannot expect to bound in advance.  Our solution will be to thin all the parts other than the escape sequences until the entire tower is so narrow that we can afford the error introduced by the various escape sequences.

So let $\tau_1,\ldots,\tau_t$ be the open loop containing $\tau$ and let $e_t$ be the height of an escape sequence for $\tau_t$.  For some finite $U$ and each $i\leq U$, let $\upsilon_i\in A^n_k$ be distinct, incomparable, determined sequences so that $\mu(A^n_k)>\sum_{i\leq U}\mu([\upsilon_i])=\sum_{i\leq t}\mu([\tau_i])>bd_{T_n}(\tau)$.  (Such sequences exist because we have ensured that $\mu(A^n_k)>a^n_k\geq bd_{T_n}(\tau)$.)  For each $\upsilon_i$ there is an escape sequence contained entirely in $A^n_k$; let $e_u$ be the maximum of the heights of these sequences.  Now choose a finite $V$ and sequences $\nu_i$, $i\leq V$, so that $\mu(B^n_k)>\sum_{i\leq V}\mu([\nu_{i}])>4bd_{T_n}(\tau)$.  (Again, such sequences exist because $\mu(B^n_k)>b^n_k\geq 4bd_{T_n}(\tau)$.)  For each $\nu_{i}$ there is an escape sequence contained entirely in $B^n_k$; let $e_v$ be the maximum of the heights of these sequences.

For each $\upsilon_i$ we fix an escape sequence entirely in $A^n_k$, and for each $\nu_i$ we fix an escape sequence entirely in $B^n_k$.  We now choose $N$ and $N'$ sufficiently large to carry out the following argument.  We apply Lemma \ref{thin_loop} to $\tau_1,\ldots,\tau_t$ with $\epsilon=2^{-(N+N')}$, and to each $\upsilon_i$ and $\nu_i$ we first apply Lemma \ref{thin_loop_with_leftover} with $\epsilon=2^{-N}$, and then, letting $\upsilon'_i$ and $\nu'_i$ be the portions which are thinned, we then apply Lemma \ref{thin_loop} to $\upsilon'_i$ and $\nu'_i$ with $\epsilon=2^{-N'}$.

We now take escape sequences of width $2^{-(N+N'-1)}$ for the $\upsilon'_i$ and $\nu'_i$.  By choosing $N$ large enough, we may ensure that we may choose these sequences to be nonoverlapping and that the leftover portions from the application of Lemma \ref{thin_loop_with_leftover} are not completely filled by these portions.  We take an escape sequence for $\tau_t$ of the same width.  We now apply Lemma \ref{escaping} repeatedly.

Let $T'$ be the partial transformation we obtain after all these applications.  $\cup_{i\leq t}[\tau_i]$ is contained in some open loop with final element $\tau^*$ and $T'(\tau^*)=\langle\rangle$ (because $\tau^*$ was the last element of the escape sequence).  Similarly $[\upsilon'_i]$ and $[\nu'_i]$ are contained in open loops with initial elements $\upsilon^0_i,\nu_i^0$ and final elements $\upsilon^*_i,\nu^*_i$.  Then we define $T_{n+1}(\tau^*)=\upsilon_0^0$, for $i<U$, $T_{n+1}(\upsilon_i^*)=\upsilon_{i+1}^0$, $T_{n+1}(\upsilon_U^*)=\nu_0^0$, and for $i<V$, $T_{n+1}(\tau_i^*)=\tau_{i+1}^0$.

Consider some element $x$ of $[\tau]$.  Let $k=2^{|\tau|}bd_{T_n}(\tau)$ be the length of the open loop in $T_n$ containing $\tau$.  In $T_{n+1}$ and any extension of $T_{n+1}$, we have ensured that $x$ is first mapped through $[\tau]$ for at most $2^{N+N'-|\tau|}k$ steps.  $x$ is then mapped through an escape sequence of length $e_t$.  $x$ is next mapped through the $[\upsilon_i]$ and their escape sequences; the $\upsilon_i$ account for
\[\sum_{i\leq U}2^{N+N'-|\upsilon_i|}-2^{N'-|\upsilon_i|}=(2^{N+N'}-2^{N'})\sum_{i\leq U}2^{-|\upsilon_i|}\geq 2^{N+N'}k2^{-|\tau|}+e_t\]
steps since $N,N'$ are sufficiently large.  In particular, this means that by the time $x$ reaches the end of the $[\upsilon_i]$, more than half these steps were in $A^n_k\subseteq A$.  Next, $x$ is mapped through the $[\nu_i]$ and their escape sequences.  Since
\[\sum_{i\leq V}2^{N+N'-|\nu_i|}-2^{N'-|\upsilon_i|}=(2^{N+N'}-2^{N'})\sum_{i\leq V}2^{-|nu_i|}>2^{N+N'}\cdot 2\cdot\sum_{i\leq U}2^{-|\upsilon_i|}+Ue_u,\]
 we have ensured that the number of steps in the $[\nu_i]$ is twice as large as the entire segment before we entered the $[\nu_i]$.  In particular, this ensures that, after leaving the $[\nu_i]$, at most one third of the steps were in $A$.

To find $A^{n+1}_k$ we remove all the $[\upsilon_i]$ and their escape sequences, and to find $B^{n+1}_k$ we remove all the $[\tau_i]$ and their escape sequences.  Naturally, we add these to $W^{n+1}_{k+1}$, along with $[\tau]$ and its escape sequence.  Now consider any element of the long open loop containing $[\tau]$, and let $\gamma$ be the burden of this element.  We choose $j$ large enough that $\mu(A^n_{k+1})-a^n_{k+1}>\gamma 2^{-j}$ and $\mu(B^n_{k+1}-b^n_{k+1})>\gamma\cdot 4\cdot 2^{-j}$, set $\rho^{n+1}$ to be $j$ for every element of this open loop, and set $a^{n+1}_{k+1}=a^n_{k+1}+\gamma 2^{-j}$ and $b^{n+1}_{k+1}=b^n_{k+1}+4\gamma 2^{-j}$.  Finally $a^{n+1}_k=a^n_k-\gamma_u$ where $\gamma_u$ is the sum of the measures of the escape sequences for the $\upsilon_i$, and $b^{n+1}_k=b^n_k-\gamma_v$ where $\gamma_v$ is the sum of the measures of the escape sequences for the $\nu_i$.  (We leave unchanged all other objects we maintain inductively---for instance, if $k'\neq k$ then $A^{n+1}_{k'}=A^n_k$ and so on.)

This completes the construction of $T_{n+1}$ and all associated objects in the case where we enumerate a single determined element into $\cup_j V_j$, as well as the proof that the construction is valid in that case.

We must confront one final complication: it may be that $\tau$ is not determined.  In this case we would like to split $[\tau]$ into a union $\cup_j[\tau^j]$ where the $\tau^j$ are determined.  To do so, we apply the construction just described to each $\tau^j$ repeatedly.  If multiple $\tau^j$ belong to the same open loop, we treat them simultaneously since the construction above applies to the entire open loop containing $[\tau]$.  The only possible source of interference is that we may alter the escape sequences of one $\tau^j$ while applying the construction to another one.  However we note that we do not need any sort of uniform bound on the length of escape sequences, and therefore this is not an obstacle.

\begin{figure}
  \begin{tikzpicture}
\draw (0,0) -- (0.5,0) -- (0.5,8.75) -- (0,8.75) -- (0,0);
\draw (0,1) -- (.5,1);
\draw[dotted] (0,0.167) -- (0.5,.167);
\draw[dotted] (0,0.333) -- (0.5,.333);
\draw[dotted] (0,0.5) -- (0.5,.5);
\draw[dotted] (0,0.667) -- (0.5,.667);
\draw[dotted] (0,0.833) -- (0.5,.833);
\draw[decoration={brace,mirror,amplitude=6pt},decorate] (0.5,0) -- (0.5,1) node[midway, right=3pt]{The open loop containing $\tau$};
\draw (0,1.75) -- (.5,1.75);
%\draw[dotted] (0,1.25) -- (.5,1.25);
\draw[decoration={brace,amplitude=6pt},decorate] (0,1) -- (0,1.75) node[midway, left=3pt]{An escape sequence for $\tau$};

\draw (0,2.75) -- (.5,2.75);
\draw (0,3.5) -- (.5,3.5);
\draw (0,4.5) -- (.5,4.5);
\draw (0,5.25) -- (.5,5.25);
\draw[dotted] (0,2) -- (.5,2);
\draw[dotted] (0,2.25) -- (.5,2.25);
\draw[dotted] (0,2.5) -- (.5,2.5);
\draw[dotted] (0,3.75) -- (.5,3.75);
\draw[dotted] (0,4) -- (.5,4);
\draw[dotted] (0,4.25) -- (.5,4.25);
\draw[decoration={brace,mirror,amplitude=6pt},decorate] (0.5,1.75) -- (0.5,2.75) node[midway, right=3pt]{};
\draw[decoration={brace,mirror,amplitude=6pt},decorate] (0.5,7) -- (0.5,8) node[midway, right=3pt]{};
\node at (3.5,6.625) {\begin{minipage}{2in}Portions of $B^n_k$ with total area slightly larger than $bd_T(\tau)$.\end{minipage}};
\draw (0.7,5.75) arc (-90:0:0.5);
\draw (0.7,7.5) arc (90:5:0.5);
\draw[decoration={brace,amplitude=6pt},decorate] (0,6.25) -- (0,7) node[midway, left=3pt]{};
\draw[decoration={brace,amplitude=6pt},decorate] (0,8) -- (0,8.75) node[midway, left=3pt]{};
\node at (-2,7.75) {\begin{minipage}{1.5in}Escape sequences for the portions of $B^n_k$.\end{minipage}};
\draw (-0.2,6.625) arc (270:180:0.5);
\draw (-0.2,8.375) arc (90:180:0.5);

\draw (0,6.25) -- (.5,6.25);
\draw (0,7) -- (.5,7);
\draw (0,8) -- (.5,8);
\draw (0,8.75) -- (.5,8.75);
\draw[dotted] (0,5.5) -- (.5,5.5);
\draw[dotted] (0,5.75) -- (.5,5.75);
\draw[dotted] (0,6) -- (.5,6);
\draw[dotted] (0,7.25) -- (.5,7.25);
\draw[dotted] (0,7.5) -- (.5,7.5);
\draw[dotted] (0,7.75) -- (.5,7.75);
\draw[decoration={brace,mirror,amplitude=6pt},decorate] (0.5,5.25) -- (0.5,6.25) node[midway, right=3pt]{};
\draw[decoration={brace,mirror,amplitude=6pt},decorate] (0.5,3.5) -- (0.5,4.5) node[midway, right=3pt]{};
\node at (3.5,3.125) {\begin{minipage}{2in}Portions of $A^n_k$ with total area slightly larger than $bd_T(\tau)$.\end{minipage}};
\draw (0.7,2.25) arc (-90:0:0.5);
\draw (0.7,4) arc (90:5:0.5);
\draw[decoration={brace,amplitude=6pt},decorate] (0,2.75) -- (0,3.5) node[midway, left=3pt]{};
\draw[decoration={brace,amplitude=6pt},decorate] (0,4.5) -- (0,5.25) node[midway, left=3pt]{};
\node at (-2,4) {\begin{minipage}{1.5in}Escape sequences for the portions of $A^n_k$.\end{minipage}};
\draw (-0.2,3.125) arc (270:180:0.5);
\draw (-0.2,4.875) arc (90:180:0.5);

  \end{tikzpicture}
\caption{Construction of $T'$}
\label{one_step_figure}
\end{figure}
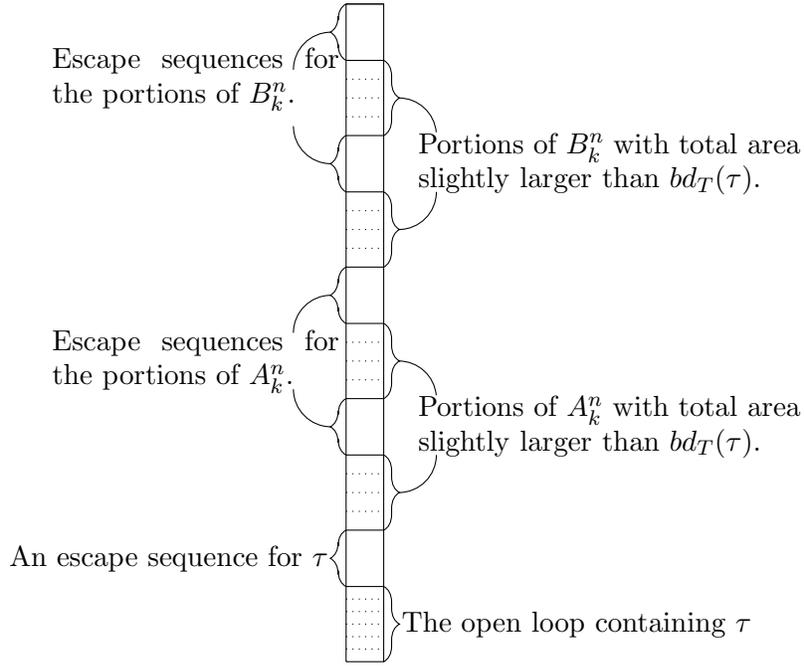
\end{proof}

\section{Upcrossings}\label{upcross}

Throughout this section, we will take $T$ to be a computable, measure-preserving transformation.

Recall the following theorem of Bishop's \cite{bishop:MR0228655}:
\begin{theorem}
\[\int \tau(x,f,\alpha,\beta)dx\leq\frac{1}{\beta-\alpha}\int(f-\alpha)^+dx.\]
\end{theorem}

This is easily used to derive the following theorem of V'yugin:
\begin{theorem}[\cite{v97}]
If $x$ is \ml random and $f$ is computable then $\lim_{n\rightarrow\infty}\frac{1}{n+1}\sum_{j=0}^nf(T^jx)$ converges.
\end{theorem}
\begin{proof}
Suppose $\lim_{n\rightarrow\infty}\frac{1}{n+1}\sum_{j=0}^nf(T^jx)$ does not converge.  Then there exist $\alpha<\beta$ such that $\frac{1}{n+1}\sum_{j=0}^nf(T^jx)$ is infinitely often less than $\alpha$ and also infinitely often greater than $\beta$.  Equivalently, $\tau(x,f,\alpha,\beta)$ is infinite.  But observe that when $f$ is computable, $\tau(x,f,\alpha,\beta)$ is lower semi-computable, so in particular,
\[V_n=\{x\mid \tau(x,f,\alpha,\beta)\geq n\}\]
is computably enumerable and $\mu(V_n)\leq \frac{1}{n(\beta-\alpha)}\int(f-\alpha)^+dx$.  Therefore an appropriate subsequence of $\langle V_n\rangle$ provides a \ml test, and $x\in\cap_n V_n$, so $x$ is not \ml random.
\end{proof}

We now consider the case where $f$ is lower semi-computable.  We will have a sequence of uniformly computable increasing approximations $f_i\rightarrow f$, and we wish to bound the number of upcrossings in $f$.  The difficulty is that $\tau(x,f_i,\alpha,\beta)$ is not monotonic in $i$: it might be that an upcrossing sequence for $f_i$ ceases to be an upcrossing sequence for $f_{i+1}$.

In order to control this change, we need a suitable generalization of upcrossings, where we consider not only the upcrossings for $f$, but for all functions between $f$ and $f+h$ where $h$ is assumed to be small.
\begin{definition}
A \emph{loose upcrossing sequence} for $\alpha,\beta,f,h$ is a sequence
\[0\leq u_1<v_1<u_2<v_2<\cdots<u_N<v_N\]
such that for all $i\leq N$,
\[\frac{1}{u_i+1}\sum_{j=0}^{u_i}f(T^jx)<\alpha,\ \frac{1}{v_i+1}\sum_{j=0}^{v_i}(f+h)(T^jx)>\beta.\]

$\upsilon(x,f,h,\alpha,\beta)$ is the supremum of the lengths of loose upcrossing sequences for $\alpha,\beta,f,h$.
\end{definition}

Loose upcrossings are much more general than we really need, and so the analog of Bishop's theorem is correspondingly weak.  For instance, consider the case where $T$ is the identity transformation, $f=\chi_A$, and $h=\chi_B$ with $A$ and $B$ disjoint (so $f+h=\chi_{A\cup B}$).  Then $\upsilon(x,f,h,\alpha,\beta)=\infty$ whenever $0<\alpha<\beta<1$. Nonetheless, we are able to show the following:
\begin{theorem}
Suppose $h\geq 0$, $\int h\ dx<\epsilon$ and $\beta-\alpha>\delta$.  There is a set $A$ with $\mu(A)<4\epsilon/\delta$ such that
\[\int_{X\setminus A}\upsilon(x,f,h,\alpha,\beta)dx\]
is finite.
\label{loose_upcrossing_lemma}
\end{theorem}
\begin{proof}
By the usual pointwise ergodic theorem, there is an $n$ and a set $A'$ with $\mu(A')<2\epsilon/\delta$ such that if $x\not\in A'$ then for all $n',n''\geq n$,
\[\left|\frac{1}{n'+1}\sum_{j=0}^{n'}h(T^jx)-\frac{1}{n''+1}\sum_{j=0}^{n''}h(T^jx)\right|<\delta/2.\]

Consider those $x\not\in A'$ such that, for some $n'\geq n$,
\[\frac{1}{n'+1}\sum_{j=0}^{n'}h(T^jx)\geq \delta.\]
We call this set $A''$.  Then for all $n'\geq n$, such an $x$ satisfies
\[\frac{1}{n'+1}\sum_{j=0}^{n'}h(T^jx)\geq \delta/2,\]
and in particular,
\[\int_{A''}h\;dx\geq \delta\mu(A'')/2.\]
Therefore $\mu(A'')\leq 2\epsilon/\delta$.  If we set $A=A'\cup A''$, we have $\mu(A)< 4\epsilon/\delta$.

Now suppose $x\not\in A$.  We claim that any loose upcrossing sequence for $\alpha,\beta,f,h$ with $n\leq u_1$ is already an upcrossing sequence for $\alpha,\beta-\delta$.  If $n\leq u_1<v_1<\cdots<u_N<v_N$ is a loose upcrossing sequence, we automatically satisfy the condition on the $u_i$.  For any $v_i$, we have
\[\beta<\frac{1}{v_i+1}\sum_{j=0}^{v_i}(f+h)(T^jx)=\frac{1}{v_i+1}\sum_{j=0}^{v_i}f(T^jx)+\frac{1}{v_i+1}\sum_{j=0}^{v_i}h(T^jx).\]
Since $\frac{1}{v_i+1}\sum_{j=0}^{v_i}h(T^jx)\leq\delta$, it follows that $\frac{1}{v_i+1}\sum_{j=0}^{v_i}f(T^jx)>\beta-\delta$ as desired.  Therefore
\[\int_{X\setminus A}\upsilon(x,f,h,\alpha,\beta)dx\leq \mu(X\setminus A)\int_{X\setminus A}n+\tau(x,f,\alpha,\beta-\delta)dx\]
is bounded.
\end{proof}

\begin{theorem}\label{weak2random}
If $x$ is weakly $2$-random and $f$ is lower semi-computable then $\lim_{n\rightarrow\infty}\frac{1}{n+1}\sum_{j=0}^nf(T^jx)$ converges.
\end{theorem}
\begin{proof}
Suppose $\lim_{n\rightarrow\infty}\frac{1}{n+1}\sum_{j=0}^nf(T^jx)$ does not converge.  Then there exist $\alpha<\beta$ such that $\frac{1}{n+1}\sum_{j=0}^nf(T^jx)$ is infinitely often less than $\alpha$ and also infinitely often greater than $\beta$.  Equivalently, $\tau(x,f,\alpha,\beta)$ is infinite.  Let $f_n\rightarrow f$ be the sequence of computable functions approximating $f$ from below.

For each $n$, we set
\[V_n=\{x\mid\exists m\geq n\ \upsilon(x,f_n,f_{m}-f_n,\alpha,\beta)\geq n\}.\]
By construction, $x\in\cap_n V_n$.  To see that $V_{n+1}\subseteq V_n$, observe that if
\[\upsilon(x,f_{n+1},f_{m}-f_{n+1},\alpha,\beta)\geq n+1\]
then there is a loose upcrossing sequence witnessing this, and it is easy to check (since the $f_n$ are increasing) that this is also a loose upcrossing sequence witnessing
\[\upsilon(x,f_{n},f_{m}-f_{n},\alpha,\beta)\geq n+1>n.\]

We must show that $\mu(V_n)\rightarrow 0$.  Fix $\delta<\beta-\alpha$ and let $\epsilon>0$ be given.  Choose $n$ to be sufficiently large that $||f_n-f||<\delta\epsilon/4$.  Then, since the $f_n$ approximate $f$ from below, clearly $\upsilon(x,f_m,f_{m'}-f_m,\alpha,\beta)\leq\upsilon(x,f_m,f-f_m,\alpha,\beta)$ for any $m'\geq m$.  By the previous theorem, there is a set $A$ with $\mu(A)<\epsilon/2$ such that $\int_{X\setminus A}\upsilon(x,f_m,f-f_m,\alpha,\beta)dx$ is bounded.  We may choose $n'\geq n$ sufficiently large that
\[B=\mu(\{x\not\in A\mid\upsilon(x,f_m,f-f_m,\alpha,\beta)\geq n'\})<\epsilon/2.\]
Then $V_{n'}\subseteq A\cup B$, so $\mu(V_{n'})\leq\epsilon$.
\end{proof}

\subsection{Room for Improvement}

It is tempting to try to improve Theorem \ref{loose_upcrossing_lemma}.  The premises of that theorem are too general and the proof is oddly ``half-constructive''---we mix the constructive and nonconstructive pointwise ergodic theorems.  One would think that by tightening the assumptions and using Bishop's upcrossing version of the ergodic theorem in both places, we could prove something stronger.

In the next theorem, we describe an improved upcrossing property which, if provable, would lead to a substantial improvement to Theorem \ref{weak2random}: balanced randomness would guarantee the existence of this limit. (Recall that a real is \emph{balanced random} if it passes every balanced test, or sequence $\langle V_i\rangle$ of r.e.\ sets such that $V_i=W_{f(i)}$ for some $2^n$-r.e.\ function $f$ and $\mu([V_i])\leq 2^{-i}$ for every $i$ \cite{fhmnn10}.) The property hypothesized seems implausibly strong, but we do not see an obvious route to ruling it out.

\begin{theorem}
  Suppose the following holds:
  \begin{quote}
    Let $f$ and $\epsilon>0$ be given, and let $0\leq h_0\leq h_1\leq \cdots\leq h_n$ be given with $||h_n||_{L^\infty}<\epsilon$.  Then
\[\int_{X}\sup_n\tau(x,f+h_n,\alpha,\beta)dx<c(||f||_{L^\infty},\epsilon)\]
where $c(||f||_{L^\infty},\epsilon)$ is a computable bound depending only on $||f||_{L^\infty}$ and $\epsilon$.
  \end{quote}
Then whenever $x$ is balanced random and $f$ is lower semi-computable then $\lim_{n\rightarrow\infty}\frac{1}{n+1}\sum_{j=0}^nf(T^jx)$ converges.
\end{theorem}
\begin{proof}
We assume $||f||_{L^2}\leq 1$ (if not, we obtain this by scaling).  Suppose $\lim_{n\rightarrow\infty}\frac{1}{n+1}\sum_{j=0}^nf(T^jx)$ does not converge.  Then there exist $\alpha<\beta$ such that $\frac{1}{n+1}\sum_{j=0}^nf(T^jx)$ is infinitely often less than $\alpha$ and also infinitely often greater than $\beta$.  Equivalently, $\tau(x,f,\alpha,\beta)$ is infinite.  Let $f_n\rightarrow f$ be the sequence of computable functions approximating $f$ from below.

We define the set
\[V_{(n,k)}=\{x\mid\exists m\geq n\ \tau(x,f_m,\alpha,\beta)\geq k\}.\]
We then define the function $g(n,n')$ to be least such that $\forall m\in[n,n']\ ||f_{n'}-f_m||<2^{-n}$ and $g(n)=\lim_{n'}g(n,n')$.  Since the sequence $f_m$ converges to $f$ from below, $g(n)$ is defined everywhere, and $|\{s\mid g(n,s+1)\neq g(n,s)\}|<2^n$ for all $n$.  Indeed, $g(n)$ is the least number such that $\forall m\geq g(n)\ ||f-f_m||\leq2^{-n}$.

Observe that $\mu( V_{(n,k)})<\frac{c(||f||_{L^\infty},2^{-n})}{k}$.  Choose $h(n)$ to be a computable function growing quickly enough that $\frac{c(||f||_{L^\infty},2^{-n})}{h(n)}\leq 2^{-n}$ for all $n$.  If $x\in V_{(g(n+1),h(n+1))}$ then there is some $m\geq g(n+1)$ so that $\tau(x,f_m,\alpha,\beta)\geq h(n+1)$.  Since $g(n+1)\geq g(n)$ and $h(n+1)\geq h(n)$, we also have that $x\in V_{(g(n),h(n))}$.  Therefore $\langle V_{(g(n),h(n))}\rangle$ is a balanced test.

But since $\tau(x,f,\alpha,\beta)$ is infinite, we must have $x\in\cap  V_{(g(n),h(n))}$. This contradicts the assumption that $x$ is balanced random, so $\lim_{n\rightarrow\infty}\frac{1}{n+1}\sum_{j=0}^nf(T^jx)$ converges.
\end{proof}
In fact, the test $\langle V_{(g(n),h(n))}\rangle$ has an additional property: if $s_0<s_1<s_2$ with $g(n+1,s_0)\neq g(n+1,s_1)\neq g(n+1,s_2)$ then $g(n,s_0)\neq g(n,s_2)$.  This means that $\langle V_{(g(n),h(n))}\rangle$ is actually an \emph{Oberwolfach test} \cite{oberwolfach_random}, and so we can weaken the assumption to $x$ being Oberwolfach random.

\section{Discussion for Ergodic Theorists}\label{ergodic_people}

In the context of analytic questions like the ergodic theorem, matters of computability are mostly questions of continuity and uniformity: the computability of a given property usually turns on whether it depends in an appropriately uniform way on the inputs.  Algorithmic randomness gives a precise way of characterizing how sensitive the ergodic theorem is to small changes in the underlying function.

The paradigm is to distinguish different sets of measure $0$, viewed as an intersection $A=\cap_i A_i$, by characterizing how the sets $A_i$ depend on the given data (in the case of the ergodic theorem, the function $f$).  The two main types of algorithmic randomness that have been studied in this context thus far are \ml randomness and Schnorr randomness.  In both cases, we ask that the sets $A_i$ be unions $A_i=\cup_j A_{i,j}$ of sets where $A_{i,j}$ is determined based on a finite amount of information about the orbit of $f$ (in particular, the dependence of $A_{i,j}$ on $f$ and $T$ should be continuous).  (To put it another way, we ask that the set of exceptional points which violate the conclusion of the ergodic theorem be contained in a $G_\delta$ which depends on $f$ in a uniform way.)  The distinction between the two notions is that in Schnorr randomness, $\mu(A_i)=2^{-i}$, while in \ml randomness, we only know $\mu(A_i)\leq 2^{-i}$.  This means that in the Schnorr random case, a finite amount of information about the orbit of $f$ suffices to limit the density of $A_i$ outside of a small set (take $J$ large enough that $\mu(\cup_{j\leq J}A_{i,j})$ is within $\epsilon$ of $2^{-i}$; then no set disjoint from $\cup_{i\leq J}A_{i,j}$ contains more than $\epsilon$ of $A_i$).  In the \ml random case, this is not possible: if $\mu(A_i)\leq 2^{-i}-\epsilon$, no finite amount of information about the orbit of $f$ can rule out the possibility that some $A_{i,j}$ with very large $j$ will add new points of measure $\epsilon$.  In particular, while we can identify sets which do belong to $A_i$, finite information about the orbit of $f$ does not tell us much about which points are not in $A_i$.

The two classes of functions discussed in this paper are the computable and the lower semi-computable ones; these are closely analogous to the continuous and lower semi-continuous functions.  Unsurprisingly, both the passage from computable to lower semi-computable functions and the passage from ergodic to nonergodic transformations make it harder to finitely characterize points violating the conclusion of the ergodic theorem.  Perhaps more surprising, both changes generate precisely the same result: if a point violates the conclusion of the ergodic theorem for a computable function with a nonergodic transformation, we can construct a lower semi-computable function with an ergodic transformation for which the point violates the conclusion of the ergodic theorem, and vice versa.

The main question we leave open is what happens when we make both changes: what characterizes the points which violate the conclusion of the ergodic theorem for lower-semi computable functions with nonergodic transformations? The answer is likely to turn on purely ergodic theoretic questions about the sensitivity of upcrossings, such as the hypothesis we use above.
\begin{question}
  Let $(X,\mu)$ be a metric space and let $T:X\rightarrow X$ be measure preserving.  Let $\epsilon>0$ be given. Is there a bound $K$ (depending on $T$ and on $\epsilon$) such that for any $f$ with $||f||_{L^\infty}\leq 1$ and any sequence $0\leq h_0\leq h_1\leq \cdots\leq h_n$ with $||h_n||_{L^\infty}<\epsilon$,
\[\int \sup_n\tau(x,f+h_n,\alpha,\beta)dx<K\]
where $\tau(x,g,\alpha,\beta)$ is the number of upcrossings from below $\alpha$ to above $\beta$ starting with the point $x$?
\end{question}

\bibliographystyle{plain}
%\bibliography{../Bibliographies/main}
\bibliography{Alg}

\end{document}